\documentclass[a4paper]{amsart}
\usepackage{amsmath, amscd}
\usepackage{amssymb, color, hyperref}

\usepackage[fontsize=11pt]{scrextend}

\theoremstyle{plain}
\newtheorem{theorem}{\bf Theorem}[section]
\newtheorem{proposition}[theorem]{\bf Proposition}

\newtheorem{corollary}[theorem]{\bf Corollary}

\theoremstyle{definition}

\newcommand{\N}{\mathbb N}
\newcommand{\Z}{\mathbb Z}
\newcommand{\R}{\mathbb R}
\newcommand{\Q}{\mathbb Q}

\newcommand{\DP}{\negthinspace : \negthinspace}

 \DeclareMathOperator{\supp}{supp}

\newcommand{\red}{{\text{\rm red}}}

\numberwithin{equation}{section}

\makeatletter
\newcommand\zeu@Scale{1.25}
\makeatother

\begin{document}

\author{Qinghai Zhong}

\address{University of Graz, NAWI Graz \\
Institute for Mathematics and Scientific Computing \\
Heinrichstra{\ss}e 36\\
8010 Graz, Austria}

\email{qinghai.zhong@uni-graz.at}
\urladdr{http://qinghai-zhong.weebly.com/}



\thanks{This work was supported by the Austrian Science Fund FWF, Project Number P28864-N35}

\begin{abstract}
Let $R$ be a Mori domain with complete integral closure $\widehat R$, nonzero conductor $\mathfrak f = (R \DP \widehat R)$, and suppose that both $v$-class groups $\mathcal C_v (R)$ and $\mathcal C_v (\widehat R)$ are finite. If $R/\mathfrak f$ is finite, then the elasticity of $R$ is either rational or infinite. If $R/\mathfrak f$ is artinian, then unions of sets of lengths of $R$ are almost arithmetical progressions with the same difference and  global bound. We derive our results in the setting of $v$-noetherian monoids.
\end{abstract}

\title{On the arithmetic of Mori monoids and domains}
\maketitle

\medskip
\section{Introduction} \label{1}
\medskip

By a $v$-noetherian monoid (or a Mori monoid),  we mean a commutative cancellative semigroup which has an identity element and which satisfies the ascending chain condition on $v$-ideals (equivalently, divisorial ideals). Thus a domain is a Mori domain if and only if its multiplicative monoid of nonzero elements is a Mori monoid.
Let $H$ be a $v$-noetherian monoid. Then every non-unit  $a \in H$ has a factorization as a finite product of atoms (irreducible elements).
If $a = u_1 \ldots u_k$, where $k \in \N$ and $u_1, \ldots, u_k$ are atoms of $H$, then $k$ is called the length of the factorization, and the set $\mathsf L (a) \subset \N$ of all possible factorization lengths  is called the set of lengths of $a$. Then $\mathcal L (H) = \{ \mathsf L (a) \mid a \in H \}$ denotes the system of sets of lengths of $H$ (with the convention that $\mathsf L (a) = \{0\}$ if $a \in H$ is invertible). The $v$-noetherian property implies that  all $L \in \mathcal L (H)$ are finite.
If there is an element $a \in H$ with $|\mathsf L (a)|>1$, then  $|\mathsf L (a^n)| > n$ for every $n \in \N$.  The structure of sets of lengths is described by a variety of invariants (see a recent survey \cite{Ge16c} and  proceedings \cite{C-F-G-O16, F-F-G-T-Z17}). In the present paper we study two of these invariants, namely unions of sets of lengths and elasticities, and we establish our results in the setting of $v$-noetherian monoids  with suitable algebraic finiteness conditions (see Theorem \ref{1.1} and Subsection \ref{examples}).

To recall the concept of unions of sets of lengths, let $k \in \N$ be given and to avoid trivial cases suppose that $H$ is not a group. Then
\[
\mathcal U_k (H) = \bigcup_{k \in L, L \in \mathcal L (H)} L \quad \text{
denotes the {\it union of sets of lengths} containing $k$}.
\]
We say that $H$ satisfies the {\it Structure Theorem for Unions} if there are $d \in \N$ and $M \in \N_0$ such that $\mathcal U_k (H)$ is an AAP (almost arithmetical progression; see \eqref{eq:defAAP}) with difference $d$ and bound $M$ for all sufficiently large $k \in \N$.
The Structure Theorem for Unions has attracted much attention in the last decade. In order to establish it in our setting we build on a recent result by Tringali (\cite{Tr18a}) and for that we need information on the growth behavior of the elasticities.

For every $k \in \N$,  $\rho_k(H)=\sup\mathcal U_k(H)$ is called the {\it $k$-elasticity}  of $H$ and
\[
\rho (H) = \sup \Big\{ \frac{\rho_k (H)}{k} \mid k \in \N \Big\} = \lim_{k \to \infty} \frac{\rho_k (H)}{k}
\]
is the {\it elasticity} of $H$. The study of the elasticities has found wide attention in the literature since the early days of factorization theory. Of course, the canonical questions are whether or not the $\rho_k (H)$s and $\rho (H)$ are finite, whether or not $\rho (H)$ is rational, or even  accepted (i.e., the supremem in the above definition is even a maximum), and to derive precise values for the $k$-th elasticities. We refer to \cite{An97a} (a survey from the late 1990s),  to \cite{Ka05a} (offering a characterization of the finiteness of the elasticity of finitely generated domains),  to \cite{Sc16a} (offering a survey on precise values of $k$th elasticities of Krull monoids),  and to \cite{Ge-Zh18a, Go19b, Zh19a} for recent progress.

Before we formulate our main result, recall that if $H$ is a
$v$-noetherian monoid  with nonempty conductor $(H \DP \widehat H)$, where $\widehat H$ is the complete integral closure of $H$,  then $\widehat H$ is Krull whence $\widehat H$ and $\widehat H_{\red}$ have a divisor theory.

\smallskip
\begin{theorem} \label{1.1}
Let $H$ be a $v$-noetherian monoid such that the conductor $(H \DP \widehat{H})\neq \emptyset$ and the class group $\mathcal C (\widehat{H})$ of $\widehat H$ is a torsion group.  Let $F=\widehat{H}^{\times}\times \mathcal F(P)$ such that  $\widehat{H}_{\red}\hookrightarrow \mathcal F(P)$ is a divisor theory and let $E=\{p\in P\mid \text{there exists }n\in \N \text{ such that } p^n\in H\widehat{H}^{\times}\}\,.$
\begin{enumerate}
\item If $H\subset F$ is not simple, then $\rho_k (H) = \infty$ for all sufficiently large $k \in \N$  and $H$ satisfies the Structure Theorem for Unions.

\item If $H\subset F$ is simple and the class semigroup $\mathcal C_{E}(H,F)$ is finite, then   $\rho(H)$ is rational  and $H$ satisfies the Structure Theorem for Unions.
\end{enumerate}
\end{theorem}

In special cases the results of Theorem \ref{1.1} are already known. To begin with, suppose that $H$ is a Krull monoid with finite class group. Then $H = \widehat H = (H \DP \widehat H)$, $E=P$, and $\mathcal C (H) \cong \mathcal C (H,F)$, and it is well-known that the elasticity of $H$ is rational and $H$ satisfies the Structure Theorem for Unions. However, even in this special case, the statements may fail without the finiteness assumption on the class group. Indeed, D.D. Anderson and D.F. Anderson showed that for every real number $r$ there is a Dedekind domain with torsion class group and elasticity equal to $r$ (\cite{An-An92}). Furthermore,  there exist locally tame Krull monoids with finite set of distances that do not satisfy the Structure Theorem for Unions(\cite[Theorem 4.2]{F-G-K-T17}). Thus, even in the case of Krull monoids, finiteness assumptions on class groups are indispensable for obtaining arithmetical finiteness results. The interesting cases of Theorem \ref{1.1}, where the results are new, are those where $H$ is not completely integrally closed.
In Subsection \ref{examples}, after having introduced more terminology, we provide a list of monoids and domains  fulfilling the abstract assumptions above.
Here we formulate one corollary.

\begin{corollary} \label{1.2}
Let $R$ be a Mori domain with nonzero conductor $\mathfrak f = (R \DP \widehat R)$ such that the $v$-class groups $\mathcal C_v (\widehat R)$ and  $\mathcal C_v (R)$ are both finite. If the factor ring $S^{-1}\widehat R/S^{-1}\mathfrak f$ is quasi-artinian $($where $S$ is the monoid of regular elements of $R)$, then  $R$ satisfies the Structure Theorem for Unions.
\end{corollary}

\newpage
\medskip
\section{Preliminaries} \label{2}
\medskip

We denote by $\N$ the set of positive integers and denote $\N_0=\N\cup\{0\}$.   For real numbers $a, b \in \R$ we denote by $[a, b] = \{ x \in \Z \mid a \le x \le b \}$ the discrete interval between $a $ and $b$.  Let $A, B \subset \Z$ be subsets. Then  $A + B = \{a+b \mid a \in A, b \in B\}$ denotes their sumset and $\Delta (A)$ denotes the \ {\it set of distances}  of $A$, that is the set of all $d \in
\N$ for which there exists $l \in A$ such that $A \cap [l, l+d] =
\{l, l+d\}$. If $\emptyset \ne A \subset \N$, then $\rho (A) = \sup \{ a/b \mid a, b \in A\} \in \Q_{\ge 1} \cup \{\infty\}$ denotes the {\it elasticity} of $A$ and we set $\rho (\{0\})=1$.
A subset  $L
\subset \mathbb Z$ \ is called an {\it almost arithmetical
progression} ({\rm AAP} for short) with {\it difference} $d \in \N$ and
{\it bound} $M \in \N_0$ if
\begin{equation} \label{eq:defAAP}
L = y + (L' \cup L^* \cup L'') \subset y + d \mathbb Z
\end{equation}
where $L^*$ is a non-empty arithmetical progression with difference $d$  such that $\min L^* = 0$, $L' \subset [-M,-1]$, $L'' \subset \sup L^* + [1, M]$ (with the convention that $L'' = \emptyset$ if $L^*$ is infinite) and $y \in \mathbb Z$.

\noindent
\subsection{\bf Monoids and domains.} \label{arithmetic}
By a monoid, we mean a commutative cancellative semigroup with identity element. Thus, if $R$ is a domain, then its multiplicative semigroup of nonzero elements is a monoid. All ideal theoretic and arithmetic concepts are introduced in the monoid setting but will be used for monoids and domains.

Let $H$ be a monoid. We denote by $\mathcal A (H)$ the set of atoms of $H$, by $H^{\times}$ the group of invertible elements of $H$,  by $H_{\red} =H/H^{\times}$ the associated reduced monoid of $H$,   by $\mathsf q(H)$ the quotient group of $H$, and by
$$\widehat{H}=\{x\in \mathsf q(H)\mid \text{ there exists $c\in H$ such that }cx^n\in H \text{ for all }n\in \N\}$$
the complete integral closure of $H$.
For subsets $X,Y\subset \mathsf q(H)$, we set
\[
(X \DP Y)=\{a\in \mathsf q(H)\mid aY\subset X\} \quad \text{ and } \quad X_{v}=\big( H \DP (H \DP X)\big) \,.
\]
The subset $X$ is called a $v$-ideal (or a divisorial ideal) of $H$ if $X\subset H$ and $X_{v}=X$. The monoid $H$ is called a {\it $v$-noetherian} monoid (or a {\it Mori monoid}) if $H$ satisfies the ascending chain condition on $v$-ideals.

If an element $a
\in H$ has a factorization into $k$ atoms, say $a=u_1  \ldots  u_k$ with $k \in \N$ and  $u_1,\ldots, u_k\in \mathcal A(H)$, then $k$ is called a factorization length of $a$ and
\[
\mathsf L (a)  = \bigl\{ k\in \N \, \bigm| \,k \text{ is a factorization length of }a \bigr\} \subset \N_0 \quad \text{is the \ {\it set of
lengths} \ of $a$}\,.
\]
For $a \in H^{\times}$ we set $\mathsf L (a) = \{0\}$. Then
\[
\mathcal L(H)=\{\mathsf L(a)\mid a\in H\}   \qquad \text{resp.} \qquad \Delta (H) = \bigcup_{L \in \mathcal L (H)} \Delta (L)
\]
denotes the {\it system of sets of lengths} resp. the {\it set of distances} of $H$. If $H$ is $v$-noetherian, then all sets of lengths are finite and nonempty (\cite[Theorem 2.2.9]{Ge-HK06a}) and if there is a set of lengths that is not a singleton, then there are arbitrarily large sets of lengths. We say that $H$ satisfies the {\it Structure Theorem for Unions} if there are $d \in \N$ and $M \in \N_0$ such that $\mathcal U_k (H)$ is an AAP  with difference $d$ and bound $M$ for all sufficiently large $k \in \N$.
If $\Delta (H)$ is finite, $\rho_k (H) < \infty $ for all $k \in \N$, and the Structure Theorem for Unions holds with difference $d$ and some bound, then, by \cite[Corollary 2.3]{F-G-K-T17},  $d = \min \Delta (H)$ and
\[
\lim_{k \to \infty} \frac{|\mathcal U_k (H)|}{k} = \frac{1}{\min \Delta (H)} \Big( \rho (H) - \frac{1}{\rho (H)} \Big) \,.
\]
Let $F=F^{\times}\times \mathcal F(P)$ be a factorial monoid, where $\mathcal F (P)$ the  {\it free
abelian monoid}  with basis $ P$, and let $a \in F$.
Then $a$ has a unique representation of the form
\[
a = \varepsilon \prod_{p \in  P} p^{\mathsf v_p(a) } \quad \text{with $\varepsilon \in F^{\times}$ and } \quad
\mathsf v_p(a) \in \N_0 \ \text{ and } \ \mathsf v_p(a) = 0 \ \text{
for almost all } \ p \in  P \,.
\]
We  call $|a|= |a|_{\mathcal F ( P)} = \sum_{p \in  P}\mathsf v_p(a) \in \N_0$ the {\it length} of $a$ and
 $\supp_P (a)=\supp(a)=\{p\in P\mid \mathsf v_p(a)\neq 0\} \subset P$ the {\it support} of $a$.

\smallskip
\subsection{\bf Krull monoids.} \label{Krull}
A monoid $H$ is a {\it Krull monoid} if it satisfies one of the following equivalent conditions (\cite[Theorem 2.4.8]{Ge-HK06a}){\rm \,:}
\begin{itemize}
\item[(a)] $H$ is completely integrally closed and $v$-noetherian.
\item[(b)] $H$ has a divisor theory.
\item[(c)] There is a divisor homomorphism from $H$ to a factorial monoid.
\end{itemize}
Let $H$ be a  Krull monoid. Then there is a factorial monoid $F = H^{\times} \times \mathcal F (P)$ such that the inclusions $H \hookrightarrow F$ and  $H_{\red} \hookrightarrow F$ are divisor theories.  The group $\mathcal C (H) = \mathsf q (F)/\mathsf q (H_{\red})$ is the (divisor) class group of $H$ and it is isomorphic to the $v$-class group $\mathcal C_v (H)$ of $H$. If $H$ is $v$-noetherian with $(H \DP \widehat H) \ne \emptyset$, then $\widehat H$ is Krull and $H^{\times} = {\widehat H}^{\times} \cap H$ (\cite[Theorem 2.3.5]{Ge-HK06a}).
Note that a domain is a Mori domain (a Krull domain, or completely integrally closed) if and only if its multiplicative monoid of nonzero elements has the respective property. We refer to \cite{Ba00} for a survey on Mori domains and to the monographs \cite{HK98, Ge-HK06a} for an exposition of Krull monoids and domains.

\smallskip
\noindent
\subsection{\bf Class semigroups.} \label{class semigroups} Class semigroups were introduced in \cite{Ge-HK04a} and we follow that notation of \cite[Chapter 2]{Ge-HK06a}.
Let $F=F^{\times}\times \mathcal F(P)$ be a factorial monoid and $H\subset F$  a submonoid with $H^{\times}=F^{\times}\cap H$. We say two elements $x,y\in F$ are $(H,F)$-equivalent if $x^{-1}H\cap F=y^{-1}H\cap F$. Then $(H,F)$-equivalence is a congruence relation on $F$. For every $x\in F$, we denote by $[y]_H^F$ the $(H,F)$-equivalence class of $x$.
For a subset $E \subset F$, we define $\mathcal C_E(H,F)=\{[x]_H^F\mid x\in E\}$.
The semigroups
\[
\mathcal C(H,F)=\mathcal C_F(H,F) \quad \text{resp.} \quad \mathcal C^*(H,F)=\mathcal C_ {(F\setminus F^{\times}\})\cup \{1\}}(H,F)
\]
are   called the {\it class semigroup of $H$ in $F$} resp. the   {\it  reduced class semigroup of $H$ in $F$}.
A subset $E \subset P$ is called {\it $H$-essential} if $E=\supp_P(x)$ for some $x\in H\setminus H^{\times}$, and $H$ is called {\it simple} in $F $ if every minimal $H$-essential subset of $P$ is a singleton.

\smallskip
\noindent
\subsection{Monoids and domains satisfying the assumptions of Theorem \ref{1.1}.} \label{examples}
A submonoid $H$ of a factorial monoid $F$ with $H^{\times}=F^{\times}\cap H$ is called a C-{\it monoid} if the reduced class semigroup $\mathcal C^*(H,F)$ is finite. A domain is said to be a C-domain if its monoid of nonzero elements is a C-monoid. Since every C-monoid $H$ is $v$-noetherian with nonempty conductor $(H \DP \widehat H)$ and finite class group $\mathcal C (\widehat H)$, every C-monoid satisfies the assumptions made in Theorem \ref{1.1}. Thus the elasticity of a C-monoid is rational whenever it is finite. For certain classes of weakly Krull C-monoids (including orders in number fields) the elasticity is even accepted whenever it is finite (\cite[Theorem 4.4]{Ge-Zh18a} but this is not true for C-monoids in general (see \cite[page 226]{Ge-HK06a}, \cite{Ba-Ch14a}, and the literature therein).

We refer to \cite[Chapter 2]{Ge-HK06a} for the basics on C-monoids and to \cite{Re13a, Ge-Ra-Re15c, Cz-Do-Ge16a} for recent progress and new classes of examples. Here we just mention the most significant example from ring theory. Let $R$ be a Mori domain with nonzero conductor $\mathfrak f = (R \DP \widehat R)$ and finite class group $\mathcal C (\widehat R)$ (note that every C-domain is such a domain). If the factor ring $R/\mathfrak f$ is finite, then $R$ is a C-domain by \cite[Theorem 2.11.9]{Ge-HK06a}.

In general, the assumption on the finiteness of $R/\mathfrak f$ cannot be omitted (indeed, for some classes of domains it is equivalent for a domain to be a C-domain). In \cite{Ge-Ha08b} and \cite{Ka16b}, weakly C-monoids were studied and this concept allows to derive arithmetical finiteness results for Mori domains with nonzero conductor $\mathfrak f = (R \DP \widehat R)$ and finite class group $\mathcal C (\widehat R)$ but without the finiteness assumption on the factor ring $R/\mathfrak f$. Weakly C-monoids provide further examples where  Theorem \ref{1.1} can be applied.
However, we do not introduce this concept here but we demonstrate its consequences for the large class of Mori domains addressed in Corollary \ref{1.2}. Here is its short proof  (based on Theorem \ref{1.1}).

\begin{proof}[Proof of Corollary \ref{1.2}]
Suppose that the given factor ring is quasi-artinian. We use Theorems 5.5, 6.2., and 7.2 of \cite{Ka16b}. If $R$ has finite tame degree $\mathsf t (R)$, then $R$ satisfies the Structure Theorem for Unions by \cite[Theorems 3.5 and  4.2]{Ga-Ge09b}. Suppose the tame degree is infinite. Then the elasticity $\rho (R)$ is infinite and the monoid $R \setminus \{0\}$ is not simple in a (suitable) factorial monoid $F$. Thus Theorem \ref{1.1}.1 implies that $R$ satisfies the Structure Theorem for Unions.
\end{proof}

\medskip
\section{Proof of Theorem \ref{1.1} }\label{3}
\medskip

 We begin with the following proposition.

\begin{proposition}\label{3.1}
Let $H$ be a $v$-noetherian monoid with $(H \DP \widehat{H})\neq \emptyset$. Let $F=\widehat{H}^{\times}\times \mathcal F(P)$ such that  $\widehat{H}_{\red}\hookrightarrow \mathcal F(P)$ is a divisor theory, let $\widetilde{F}=\widehat{H}^{\times}/H^{\times}\times \mathcal F(\widetilde{P})$ with $\widetilde{P}=\{[p]_H^F\mid p\in P\}$, and let $\beta \colon \mathsf q(F)\rightarrow \mathsf q(\widetilde{F})$ be the unique homomorphism satisfying $\beta(p)=[p]_H^F$ for all $p\in P$ and $\beta(u)=uH^{\times}$ for all $u\in \widehat{H}^{\times}$.
\begin{enumerate}
\item  $\beta(H)$ is a reduced  submonoid of $\widetilde{F}$, $\beta(H)\cap \widetilde{F}=\{1\}$, $\beta^{-1}(\beta(H))=H$, and $\mathcal L(H)=\mathcal L(\beta(H))$.

\item  $\beta(\mathsf q(H))=\mathsf q(\beta(H))$ and   $\beta(H)$ is  $v$-noetherian.

\item $\widehat{\beta(H)}=\beta(\widehat{H})$ and $(\beta(H) \DP \widehat{\beta(H)})\neq \emptyset$.

\item $\widehat{\beta(H)}^{\times}=\widehat{H}^{\times}/H^{\times}$ and $\widehat{\beta(H)}_{\red}\hookrightarrow\mathcal F(\widetilde{P})$ is a divisor theory.

\item
\begin{align*}
&\{[p]_H^F\mid p\in P\text{ and there exists $n\in \N$ and $u\in \widehat{H}^{\times}/H^{\times}$ such that }u\beta(p)^n\in \beta(H)\}\\
=&\{[p]_H^F\mid p\in P\text{ and there exists $n\in \N$ and $u\in \widehat{H}^{\times}$ such that }up^n\in H\}\,.
\end{align*}

\end{enumerate}
\end{proposition}
\begin{proof}
1.  follows from \cite[Theorem 3.2.8 and Proposition 3.2.3.5]{Ge-HK06a}.

2. $\beta(\mathsf q(H))=\mathsf q(\beta(H))$ is clear. Let $E\subset \beta(H)$. Then there exists $E'\subset H$ such that $\beta(E')=E$. Since $H$ is $v$-noetherian, it follows by \cite[Proposition 2.1.10.1]{Ge-HK06a} that there exists a finite subset $E_0\subset E'$ such that $(H \DP E')=(H \DP E_0)$. Thus $$(\beta(H) \DP E)=\beta((H \DP E'))=\beta((H \DP E_0))=(\beta(H) \DP  \beta(E_0))\,.$$
It follows by \cite[Proposition 2.1.10.1]{Ge-HK06a} again that $\beta(H)$ is $v$-noetherian.

3. Let $x\in \beta(\widehat{H})$. Then there exists $y\in \widehat{H}$ such that $x=\beta(y)\in \beta(\mathsf q(H))=\mathsf q(\beta(H))$. Since $y\in \widehat{H}$, there exist $c\in H$ such that $cy^n\in H$ for all $n\in \N$. Therefore $\beta(c)\in \beta(H)$ and $\beta(c)\beta(y)^n\in \beta(H)$ for all $n\in \N$. It follows that $x=\beta(y)\in \widehat{\beta(H)}$.

Let $x\in \widehat{\beta(H)}$. Then there exists $c'\in \beta(H)$ such that $c'x^n\in \beta(H)$ for all $n\in \N$. Since there exist $y\in \mathsf q(H)$ and $c\in H$ such that $\beta(y)=x$ and $\beta(c)=c'$, we have $\beta(cy^n)\in \beta(H)$ for all $n\in \N$. Therefore $cy^n\in \beta^{-1}(\beta(cy^n))\subset \beta^{-1}(\beta(H))=H$ for all $n\in \N$. It follows that $y\in \widehat{H}$ and $x=\beta(y)\in \beta(\widehat{H})$.

Therefore $\widehat{\beta(H)}=\beta(\widehat{H})$ and $(\beta(H) \DP \widehat{\beta(H)})=(\beta(H):\beta(\widehat{H}))=\beta((H \DP \widehat{H}))\neq \emptyset$.

4. Note that $\beta(\widehat{H}^{\times})\subset \beta(\widehat{H})\subset \widetilde{F}$. Then $\beta(\widehat{H}^{\times})\subset \beta(\widehat{H})^{\times}\subset \widetilde{F}^{\times}=\widehat{H}^{\times}/H^{\times}=\beta(\widehat{H}^{\times})$ which implies that $\beta(\widehat{H})^{\times}=\beta(\widehat{H}^{\times})$.

By \cite[Theorem 2.4.8.2]{Ge-HK06a}, we obtain $\widehat{H}\cong \widehat{H}^{\times}\times \widehat{H}_{\red}$ and $\beta(\widehat{H})\cong\beta(\widehat{H})^{\times}\times \beta(\widehat{H})_{\red}$.  Therefore $\beta(\widehat{H})^{\times}\times \beta(\widehat{H})_{\red}\cong\beta(\widehat{H}^{\times})\times \beta(\widehat{H}_{\red})$ and hence $ \beta(\widehat{H})_{\red}\cong\beta(\widehat{H}_{\red})$. Since $\widehat{H}_{\red}\hookrightarrow \mathcal F(P)$ is a divisor theory, we obtain that $\beta(\widehat{H})_{\red}\cong\beta(\widehat{H}_{\red})\hookrightarrow \beta(\mathcal F(P))=\mathcal F(\widetilde{P})$ is a divisor theory

5. Let $p\in P$, $n\in \N$ and $u\in \widehat{H}^{\times}$. Since $\beta^{-1}(\beta(H))=H$, we have $up^n\in H$ if and only if $\beta(u)\beta(p)^n\in \beta(H)$. The assertion follows immediately.
\end{proof}

\smallskip
Let $H$ be a monoid and let $S\subset H_{\red}\setminus \{1_{H_{\red}}\}$ be a generating set  of $H_{\red}$. The  monoid  $\mathsf Z^S (H) = \mathcal F \bigl(
S\bigr)$  is called the  {\it factorization
monoid}  of $H$ {\it with respect to $S$}, and  the unique homomorphism
\[
\pi^S \colon \mathsf Z^S (H) \to H_{\red} \quad \text{satisfying} \quad
\pi^S (u) = u \quad \text{for each} \quad u \in S
\]
is  the  {\it factorization homomorphism}  of $H$ with respect to $S$.
If $S=\mathcal A(H_{\red})$, then we set
 $\mathsf Z (H) :=\mathsf Z^S(H)$,  $\pi^S=\pi$, and we observe that $\mathsf L(a)=\{|z|\mid z\in \mathsf Z(H)\text{ with }\pi(z)=aH^{\times}\}$.

\begin{proof}[Proof of Theorem \ref{1.1}]
To recall the notations of Theorem \ref{1.1}, let
$H$ be a $v$-noetherian monoid with $(H \DP \widehat{H})\neq \emptyset$ and suppose that $\mathcal C (\widehat H)$ is a torsion group. Let $F=\widehat{H}^{\times}\times \mathcal F(P)$ such that $\widehat{H}_{\red}\hookrightarrow \mathcal F(P)$ is a divisor theory,   and let $E=\{p\in P\mid \text{ there exists }n\in \N \text{ such that } p^n\in H\widehat{H}^{\times}\}$.

1. Suppose  $H \subset F$ is not simple. We proceed in two steps. First we show that $\rho_k (H) = \infty$ for all sufficiently large $k \in \N$. In a second step we show that the sets of distances $\Delta(\mathcal U_k(H))$ of all $\mathcal U_k (H)$ are finite for all sufficiently large $k \in \N$. These two results imply, by \cite[Theorem 2.20]{Tr18a}, that $H$ satisfies the Structure Theorem for Unions (Note that the fact $\rho_k (H) = \infty$ for all sufficiently large $k \in \N$  implies that for any fixed $\ell\in \N$, the interval  $[\rho_{k-\ell}, \rho_k]$ is empty for  all sufficient large $k\in \N$).

1.(a) Since $H \subset F$ is not simple,  there exists $a\in \mathcal A(H)$ such that $|\supp_P(a)|\ge 2$ and $\supp_P(a)$ is a minimal $H$-essential subset. Since \cite[Proposition 2.4.5 and Theorem 2.4.7]{Ge-HK06a} implies that
 $\mathsf q(F)/\mathsf q(\widehat{H})=\mathcal C_{v}(\widehat{H})$ is a torsion group, it follows that for every $p\in \supp_P(a)$, there is an $\alpha_p\in \N$ such that $p^{\alpha_p}\mathsf q(\widehat{H})=\mathsf q(\widehat{H})$ and hence $p^{\alpha_p}\in \widehat{H}$.

Set $\alpha=\prod_{p\in \supp_P(a)}\alpha_p$ and suppose $a=u\prod_{p\in \supp_P(a)}p^{\mathsf v_p(a)}$, where $u\in \widehat{H}^{\times}$.
By the definition of complete integral closure, for every $p\in \supp_P(a)$, there exists $a_p\in H$ such that $a_p(p^{\alpha\mathsf v_p(a)})^n\in H$ holds for all $n\in \N$.
Fix  $p_0\in \supp_P(a)$. Since $u^{\alpha}p_0^{\alpha \mathsf v_{p_0}(a)}\in \widehat{H}$, there exist $b_0\in H$ such that $b_0(u^{n\alpha}p_0^{\alpha \mathsf v_{p_0}(a)})^n\in H$ holds for all $n\in H$.
It follows that for every $n\in \N$,
\begin{align*}
W_n=&b_0u^{n\alpha}p_0^{n\alpha \mathsf v_{p_0}(a)}\cdot\prod_{p\in \supp_P(a)\setminus\{p_0\}}a_pp^{n\alpha\mathsf v_p(a)}\\
= &(u^{n\alpha}\prod_{p\in \supp_P(a)}p^{n\alpha\mathsf v_p(a)})\cdot (b_0\prod_{p\in \supp_P(a)\setminus\{p_0\}}a_p)\\
=&a^{n\alpha}\cdot(b_0\prod_{p\in \supp_P(a)\setminus\{p_0\}}a_p)\in H\,.
\end{align*}
Since $|\supp_P(a)|\ge 2$ and $\supp_P(a)$ is a minimal $H$-essential subset, we obtain
$$\min\mathsf L\left(b_0u^{n\alpha}p_0^{n\alpha \mathsf v_{p_0}(a)}\right)\le |b_0|_{F}\text{ and }\min \mathsf L(a_pp^{n\alpha\mathsf v_p(a)})\le |a_p|_{F}\text{ for every }p\in \supp_P(a)\setminus\{p_0\}\,.$$

Therefore $\min \mathsf L(W_n)\le |b_0|_{F}+\sum_{p\in \supp_P(a)\setminus\{p_0\}}|a_p|_{F}$ while $\max\mathsf L(W_n)\ge n\alpha$.
Let $k\ge N_0=|b_0|_{F}+\sum_{p\in \supp_P(a)\setminus\{p_0\}}|a_p|_{F}$.
We infer that
\[
\rho_k(H)\ge \rho_{N_0}(H)\ge \rho_{\min\mathsf L(W_n)}(H)\ge n\alpha \text{ for all } n\in \N\,,
\]
whence $\rho_{k}(H)=\infty$ and $\rho(H)=\infty$.

\smallskip
1(b). Since $W_{n+1}=a^{\alpha}W_{n}$ for all $n\in \N$, we obtain that
\[
\{\min\mathsf L(W_i)+(n-i)\alpha\mid i\in [1,n]\}\subset \mathsf L(W_n) \text{ for all }n\in \N\,.
\]

Let $k\ge N_0$ and let  $N_1= \min \mathsf L(b_0\prod_{p\in \supp_P(a)\setminus\{p_0\}}a_p)$.
 Since $k\in k-\min\mathsf L(W_n)+\mathsf L(W_n)\subset \mathsf L(a^{k-\min\mathsf L(W_n)}W_n)$ and $N_1+n\alpha\in \mathsf L(W_n)$ for all $n\in \N$, we infer that
$$c_n:=k-\min\mathsf L(W_n)+N_1+n\alpha\in k-\min\mathsf L(W_n)+\mathsf L(W_n)\subset \mathcal U_{k}(H)$$ for all $n\in \N$.
Let $n_0$ be the maximal integer such that $k-(\min\mathsf L(W_1)+(n_0-1)\alpha)\ge 0$. Then $k\in k-(\min\mathsf L(W_1)+(n_0-1)\alpha)+\mathsf L(W_{n_0})\subset \mathsf L(a^{k-(\min\mathsf L(W_1)+(n_0-1)\alpha)}W_{n_0})$ which implies that
\[
d_i:=k-(i-1)\alpha+\min\mathsf L(W_i)-\min\mathsf L(W_1)\in \mathcal U_k(H) \text{ for all } i\in [1,n_0]\,.
\]

It follows that
\begin{align*}
\max\Delta(\mathcal U_k(H))&\le \max\left\{d_{n_0}, \max\{d_{i}-d_{i+1}\mid i\in [1, n_0-1]\}, c_1-d_1, \max\{d_{n+1}-d_n\mid n\in \N\}\right\} \\
&\le \max\left\{\alpha+N_0, \alpha+N_0, -\min\mathsf L(W_1)+N_1+\alpha, N_0+\alpha\right\} \\
&\le N_0+N_1+\alpha\,.
\end{align*}
Thus the sets $\Delta (\mathcal U_k (H))$ are finite for all $k \ge N_0$.


\medskip
2.  Suppose that  $H \subset F$ is simple, $\mathcal C_{v}(\widehat{H})$ is a torsion group, and $\mathcal C_E (H,F)$ is finite. Again we proceed in two steps. First we show that $\rho (H) \in \Q$. In the second step we show that there is a constant $M$ such that
\[
\rho_{k+1}(H) - \rho_k (H) \le M \quad \text{and} \quad \lambda_k (H) - \lambda_{k+1} (H) \le M \,.
\]
These two results imply, by \cite[Theorem 1.1]{Tr18a}, that $H$ satisfies the Structure Theorem for Unions.

To fix notation, let $\widetilde{F}=\widehat{H}^{\times}/H^{\times}\times \mathcal F(\widetilde{P})$ with $\widetilde{P}=\{[p]_H^F\mid p\in P\}$ and let $\beta \colon \mathsf q(F)\rightarrow \mathsf q(\widetilde{F})$ be the unique homomorphism satisfying $\beta(p)=[p]_H^F$ for all $p\in P$ and $\beta(u)=uH^{\times}$ for all $u\in \widehat{H}^{\times}$.
By Proposition \ref{3.1}, it suffices to prove the assertions for $\beta(H)$. Since $\mathcal C_E(H,F)$ is finite, to simplify the notation,
we may show the assertions for $H$ with  $E$ is finite.

2.(a)
We define that
\begin{align*}
\phi \colon F=F^{\times}\times \mathcal F(P)&\rightarrow\mathcal F(E)\\
 a=w\prod_{p\in P}p^{\mathsf v_p(a)}     &\mapsto       \phi(a)= \prod_{p\in E}p^{\mathsf v_p(a)}\,.
\end{align*}
Then $\phi(H)$ is a submonoid of $\mathcal F(E)$ and $\phi(\mathcal A(H))\subset \phi(H)\setminus \{1\}$ is a generating set of $\phi(H)$.

Since $H$ is $v$-noetherian with $(H \DP \widehat{H})\neq \emptyset$ and $E$ is finite, it follows by \cite[Theorem 4.2.2]{Ge-Ha08b} and \cite[Lemma 5.4.1]{Ge-Ha08a}  that $\sup\{\mathsf v_p(u)\mid p\in E \text{ and }u\in \mathcal A(H)\}<\infty$ which implies that
 $S=\phi(\mathcal A(H))$ is a finite set. Then $\mathsf Z^S(\phi(H))\times\mathsf Z^S(\phi(H))$  is a finitely generated monoid.
Let
$$H_0=\{(x,y)\in\mathsf Z^S(\phi(H))\times\mathsf Z^S(\phi(H))\mid \pi^S(x)=\pi^S(y)\}\,,$$
and hence $H_0\hookrightarrow \mathsf Z^S(\phi(H))\times\mathsf Z^S(\phi(H))$ is a divisor homomorphism. It follows by \cite[Proposition 2.7.5]{Ge-HK06a} that $H_0$ is a finitely generated monoid. For each $(x,y)\in H_0\setminus \{1_{H_0}\}$, we have both $|x|_{\mathsf Z^S(\phi(H))}\neq 0$ and $|y|_{\mathsf Z^S(\phi(H))}\neq 0$. Therefore
\[\begin{aligned}
\sup\left\{\frac{|y|_{\mathsf Z^S(\phi(H))}}{|x|_{\mathsf Z^S(\phi(H))}}\mid (x,y)\in H_0\setminus \{1_{H_0}\}\right\}&=\sup\left\{\frac{|y|_{\mathsf Z^S(\phi(H))}}{|x|_{\mathsf Z^S(\phi(H))}}\mid (x,y)\in \mathcal A(H_0)\right\}\\
&=\max\left\{\frac{|y|_{\mathsf Z^S(\phi(H))}}{|x|_{\mathsf Z^S(\phi(H))}}\mid (x,y)\in \mathcal A(H_0)\right\}\,.
\end{aligned}
\]

Let $(x_0,y_0)\in \mathcal A(H_0)$ such that $\frac{|y_0|_{\mathsf Z^S(\phi(H))}}{|x_0|_{\mathsf Z^S(\phi(H))}}=\sup\left\{\frac{|y|_{\mathsf Z^S(\phi(H))}}{|x|_{\mathsf Z^S(\phi(H))}}\mid (x,y)\in H_0\setminus \{1_{H_0}\}\right\}$.
We define that
\begin{align*}
\psi \colon \mathsf Z(H)&\rightarrow \mathsf Z^S(\phi(H))\\
      z=\prod_{i=1}^{n}u_i    &\rightarrow          \psi(z)=\prod_{i=1}^n\phi(u_i)\,.
\end{align*}

Thus for any $z\in \mathsf Z(H)$, the definition of $E$ implies that $|z|_{\mathsf Z(H)}=|\psi(z)|_{\mathsf Z^S(\phi(H))}$.
It follows that
\begin{align*}
\rho(H)&=\sup\left\{\frac{|z_2|_{\mathsf Z(H)}}{|z_1|_{\mathsf Z(H)}}\mid (z_1,z_2)\in \mathsf Z(H)\times \mathsf Z(H)\setminus \{1_{\mathsf Z(H)\times \mathsf Z(H)}\}\text{ with }\pi(z_1)=\pi(z_2)\right\}\\
&=\sup\left\{\frac{|\psi(z_2)|_{\mathsf Z^S(\phi(H))}}{|\psi(z_1)|_{\mathsf Z^S(\phi(H))}}\mid (z_1,z_2)\in \mathsf Z(H)\times \mathsf Z(H)\setminus \{1_{\mathsf Z(H)\times \mathsf Z(H)}\}\text{ with }\pi(z_1)=\pi(z_2)\right\}\\
&\le \sup\left\{\frac{|y|_{\mathsf Z^S(\phi(H))}}{|x|_{\mathsf Z^S(\phi(H))}}\mid (x,y)\in H_0\setminus \{1_{H_0}\}\right\}=\frac{|y_0|_{\mathsf Z^S(\phi(H))}}{|x_0|_{\mathsf Z^S(\phi(H))}}\,.
\end{align*}

Let $a,b\in H$ such that there exist $z_1,z_2\in \mathsf Z(H)$ satisfying that $\pi(z_1)=aH^{\times}$, $\pi(z_2)=bH^{\times}$, and
$(\psi(z_1), \psi(z_2))=(x_0,y_0)$.  suppose \begin{align*}
E_0&=\big(\supp_P(a)\cup \supp_P(b)\big)\setminus E=\{q_1, \ldots, q_{t}\}\subset P,\\
a&=w_1\phi(a)q_1^{\ell_1}\ldots q_t^{\ell_t},\\
b&=w_2\phi(b)q_1^{k_1}\ldots q_t^{k_t},\\
&\text{ where $\ell_1,\ldots, \ell_t, k_1,\ldots, k_t\in \N_0$ and $w_1,w_2\in \widehat{H}^{\times}$. }
\end{align*}

Since \cite[Proposition 2.4.5 and Theorem 2.4.7]{Ge-HK06a} implies that
 $\mathsf q(F)/\mathsf q(\widehat{H})=\mathcal C_{v}(\widehat{H})$ is a torsion group, it follows that  for every $i\in[1,t]$, there is an $\alpha_i\in \N$ such that $q_i^{\alpha_i}\mathsf q(\widehat{H})=\mathsf q(\widehat{H})$ and hence $q_i^{\alpha_i}\in \widehat{H}$.
Set $\alpha=\prod_{i\in [1,t]}\alpha_i$. By the definition of complete integral closure, for every $i\in [2,t]$, there exists $a_i\in H$ such that $a_iq_i^{n\alpha}\in H$ for all $n\in \N$ and there exist $b_1,b_2\in H$ such that $b_1(w_1^{\alpha}q_1^{\alpha\ell_1})^{n}\in H$ and $b_2(w_2^{\alpha}q_1^{\alpha k_1})^{n}\in H$ for all $n\in \N$.

Therefore for every $n\in \N$, we have
\begin{align*}
W_n=&a^{n\alpha}\cdot b_1\cdot b_2w_2^{n\alpha}q_1^{n\alpha k_1}\cdot \prod_{i=2}^ta_iq_i^{n\alpha k_i}\\
=& b_1b_2\prod_{i=2}^ta_i\cdot w_1^{n\alpha}w_2^{n\alpha}\phi(a)^{n\alpha}\prod_{i=1}^tq_i^{n\alpha(\ell_i+k_i)}\\
=&b^{n\alpha}\cdot b_2\cdot b_1w_1^{n\alpha}q_1^{n\alpha \ell_1}\cdot \prod_{i=2}^ta_iq_i^{n\alpha \ell_i}\in H\,.
\end{align*}
Since $\max\mathsf L(b_1\cdot b_2w_2^{n\alpha}q_1^{n\alpha k_1}\cdot \prod_{i=2}^ta_iq_i^{n\alpha k_i})\le |b_1|_F+|b_2|_F+\sum_{i=2}^t|a_i|_F$, we have
$$\frac{|y_0|_{\mathsf Z^S(\phi(H))}}{|x_0|_{\mathsf Z^S(\phi(H))}}\ge \lim_{n\rightarrow \infty}\rho( \mathsf L (W_n))\ge \lim_{n\rightarrow \infty}\frac{n\alpha|y_0|_{\mathsf Z^S(\phi(H))}}{n\alpha|x_0|_{\mathsf Z^S(\phi(H))}+|b_1|_F+|b_2|_F+\sum_{i=2}^t|a_i|_F}=\frac{|y_0|_{\mathsf Z^S(\phi(H))}}{|x_0|_{\mathsf Z^S(\phi(H))}}\,.$$
It follows that $\rho(H)=\frac{|y_0|_{\mathsf Z^S(\phi(H))}}{|x_0|_{\mathsf Z^S(\phi(H))}}$ is rational.

2(b). Set $k_0=|b_1|_F+|b_2|_F+\sum_{i=2}^t|a_i|_F$ and $M=\alpha|y_0|_{\mathsf Z^S(\phi(H))}+2k_0\frac{|y_0|_{\mathsf Z^S(\phi(H))}}{|x_0|_{\mathsf Z^S(\phi(H))}}$. For every $k\in \N$ with $k\ge \alpha|x_0|_{\mathsf Z^S(\phi(H))}+k_0$, choose $n_k\in \N$ such that
$$\{k,k+1\}\subset [n_k\alpha|x_0|_{\mathsf Z^S(\phi(H))}+k_0, (n_k+1)\alpha|x_0|_{\mathsf Z^S(\phi(H))}+k_0]\,.$$
 Since
$\rho_{n_k\alpha|x_0|_{\mathsf Z^S(\phi(H))}+k_0}(H)\ge \max\mathsf L(W_{n_k})\ge n_k\alpha|y_0|_{\mathsf Z^S(\phi(H))}$, we have
\begin{align*}
\rho_{k+1}(H)-\rho_k(H)&\le \rho_{(n_k+1)\alpha|x_0|_{\mathsf Z^S(\phi(H))}+k_0}(H)-\rho_{n_k\alpha|x_0|_{\mathsf Z^S(\phi(H))}+k_0}(H)\\
&\le ((n_k+1)\alpha|x_0|_{\mathsf Z^S(\phi(H))}+k_0)|y_0|_{\mathsf Z^S(\phi(H))}/|x_0|_{\mathsf Z^S(\phi(H))}-n_k\alpha|y_0|_{\mathsf Z^S(\phi(H))}\\
&=\alpha|y_0|_{\mathsf Z^S(\phi(H))}+k_0|y_0|_{\mathsf Z^S(\phi(H))}/|x_0|_{\mathsf Z^S(\phi(H))}\\
&\le M\,.
\end{align*}

 Fix $n_0\in \N$ such that $n_0\alpha|y_0|_{\mathsf Z^S(\phi(H))}>k_0$.
 Since $\max\mathsf L(b_2\cdot b_1w_1^{n\alpha}q_1^{n\alpha \ell_1}\cdot \prod_{i=2}^ta_iq_i^{n\alpha \ell_i})\le k_0$,
   there is a $t_n\in [0, k_0]$ such that $nn_0\alpha|y_0|_{\mathsf Z^S(\phi(H))}+t_n\in \mathsf L(W_{nn_0})$  for every $n\in \N$. For each $k\in \N$ with $k\ge n_0\alpha|y_0|_{\mathsf Z^S(\phi(H))}+k_0$, choose $n_k\in \N$ such that
   $$\{k,k+1\}\subset [n_kn_0\alpha|y_0|_{\mathsf Z^S(\phi(H))}+t_{n_k}, (n_k+1)n_0\alpha|y_0|_{\mathsf Z^S(\phi(H))}+t_{n_k+1}]\,.$$
   Then
   \begin{align*}
   \lambda_k(H)&\le \lambda_{n_kn_0\alpha|y_0|_{\mathsf Z^S(\phi(H))}+t_{n_k}}(H)+(k-n_kn_0\alpha|y_0|_{\mathsf Z^S(\phi(H))}-t_{n_k})\,.\\
   \lambda_{k+1}(H)&\ge \lambda_{(n_k+1)n_0\alpha|y_0|_{\mathsf Z^S(\phi(H))}+t_{n_k+1}}(H)-((n_k+1)n_0\alpha|y_0|_{\mathsf Z^S(\phi(H))}+t_{n_k+1}-(k+1))\,.
   \end{align*}
    Since
$\lambda_{n_kn_0\alpha|y_0|_{\mathsf Z^S(\phi(H))}+t_{n_k}}(H)\le \min\mathsf L(W_{n_kn_0})\le n_kn_0\alpha|x_0|_{\mathsf Z^S(\phi(H))}+k_0$, we have
\begin{align*}
&\quad \lambda_{k}(H)-\lambda_{k+1}(H)\\
&\le \lambda_{n_kn_0\alpha|y_0|_{\mathsf Z^S(\phi(H))}+t_{n_k}}(H)-\lambda_{(n_k+1)n_0\alpha|y_0|_{\mathsf Z^S(\phi(H))}+t_{n_k+1}}(H)+n_0\alpha|y_0|_{\mathsf Z^S(\phi(H))}+t_{n_k+1}-t_{n_k}-1\\
&\le  n_kn_0\alpha|x_0|_{\mathsf Z^S(\phi(H))}+k_0-((n_k+1)n_0\alpha|y_0|_{\mathsf Z^S(\phi(H))}+t_{n_k+1})|x_0|_{\mathsf Z^S(\phi(H))}/|y_0|_{\mathsf Z^S(\phi(H))}+k_0\\
&=2k_0-n_0\alpha|x_0|_{\mathsf Z^S(\phi(H))}-t_{n_k+1}|x_0|_{\mathsf Z^S(\phi(H))}/|y_0|_{\mathsf Z^S(\phi(H))}\\
&\le 2k_0\\
&\le M\,.\qedhere
\end{align*}
\end{proof}

\section*{Acknowledgements}

The author would like to thank Alfred Geroldinger from University of Graz   for reading the manuscript very carefully and affording very important and useful comments.

\providecommand{\bysame}{\leavevmode\hbox to3em{\hrulefill}\thinspace}
\providecommand{\MR}{\relax\ifhmode\unskip\space\fi MR }
\providecommand{\MRhref}[2]{%
  \href{http://www.ams.org/mathscinet-getitem?mr=#1}{#2}
}
\providecommand{\href}[2]{#2}

\end{document}